\documentclass{amsart}
\usepackage{amsmath}
\usepackage{amssymb}
\usepackage{amscd}
\usepackage[all]{xy}
\usepackage{color}
\usepackage{tikz-cd}
\usepackage{ulem}
\usepackage{tikz}
\usepackage{comment}
\usepackage{hyperref}
\usepackage{graphicx}

\mathchardef\mhyphen="2D

\newtheorem{theorem}{Theorem}[section]
\newtheorem{lemma}[theorem]{Lemma}
\newtheorem{corollary}[theorem]{Corollary}
\newtheorem{proposition}[theorem]{Proposition}

\newtheorem{example}[theorem]{Example}

\begin{document}
\title[On $w$-Hilbert domains]
{On $w$-Hilbert domains}

\author [H. Baek] {Hyungtae Baek}
\address{(Baek) Department of Mathematics,
Jeonbuk National University,
567 Baekje-daero, Deokjin-gu, Jeonju-si, Jeollabuk-do 54896,
Republic of Korea}
\email{htbaek5@jbnu.ac.kr}

\thanks{Words and phrases: $w$-Hilbert domain, Hilbert domain, polynomial ring, Anderson ring}

\thanks{$2020$ Mathematics Subject Classification: 13A15}

\begin{abstract}
In this paper, we introduce the notion of a $w$-Hilbert domain and investigate its basic properties.
More precisely, we explore its relationship with
Hilbert domains, strong Mori domains, and UMT domains
by providing various examples using $D+M$ constructions.
Furthermore, we characterize the polynomial rings and some of their extensions that are $w$-Hilbert domains.
\end{abstract}

\maketitle

\section{Introduction}

The notion of Hilbert rings, also known as Jacobson rings,
plays an important role in commutative algebra and algebraic geometry.
More precisely, the notion of Hilbert rings originates from Hilbert's Nullstellensatz,
which describes the relationship between radical ideals and maximal ideals in the polynomial ring over an algebraically closed field.
Motivated by Hilbert's Nullstellensatz,
Goldman and Krull introduced the notion of a ring in which
every prime ideal is an intersection of maximal ideals
\cite{G 1951, K 1951}.
Goldman referred to such rings as {\it Hilbert rings},
whereas Krull used the term {\it Jacobson rings}.
Following these foundational works, the theory of Hilbert rings was extensively refined by prominent commutative algebraists.
In particular, Kaplansky further developed this approach by introducing
the notions of G-ideals and G-domains to characterize such rings \cite{kaplansky book}.
Gilmer investigated polynomial rings over Hilbert rings,
including those in infinitely many indeterminates,
and obtained several results on when such rings are Hilbert rings \cite{G 1971}.
In \cite{AAM 1985, BH 1980}, the authors investigated the conditions under which
certain quotient extensions of the polynomial ring become Hilbert rings.
In 1992, Anderson, Dobbs and Fontana studied when pullbacks are Hilbert rings.

In this paper, we introduce the notion of $w$-Hilbert domains
motivated by the above studies, and
investigate their examples and properties.
Since every prime ideal of a Hilbert ring can be represented by maximal ideals,
maximal ideals play a central role in the notion of Hilbert rings.
Motivated by this observation,
it is natural to define a new notion combining the $w$-operation and the notion of Hilbert rings using maximal $w$-ideals as follows:
an integral domain is called a {\it $w$-Hilbert domain} if
every prime $w$-ideal can be expressed as an intersection of maximal $w$-ideals.

This paper consists of three sections including introduction.
In Section \ref{sec 2},
we investigate basic properties and some examples of $w$-Hilbert domains.
More precisely, we first show that every $w$-G-ideal of a $w$-Hilbert domain is a maximal $w$-ideal (Proposition \ref{w-G-ideal is maximal}),
and we also prove that
in a strong Mori domain,
the quotient extension of a $w$-Hilbert domain is a $w$-Hilbert domain
(Proposition \ref{quotient extension}).
We also derive some examples of $w$-Hilbert domains using $D+M$ constructions.
In more detail,
we examine an example which shows that a Hilbert ring does not imply a $w$-Hilbert domain
(Example \ref{non w-Hil Hil}), and
an example which shows that a UMT domain and a $w$-Hilbert domain do not imply each other (Example \ref{UMT and w-Hil}).
In Section \ref{sec 3},
we first characterize the polynomial rings that are $w$-Hilbert domains (Theorem \ref{w-Hilbert poly}).
After the characterization,
we examine the construction of some extensions of the polynomial rings, and
we also characterize such extensions that are $w$-Hilbert domains (Theorems \ref{w-Hilbert Anderson} and \ref{w-Hilbert t-Nagata}).
We then construct an example of a $w$-Hilbert domain that is neither a Hilbert domain, a DW-domain, nor a DVR (Example \ref{w-Hilbert nor DVR}).

To help readers better understand this paper,
we review the definitions of the $w$-operation and
some integral domains related to the $w$-operation.
In this paper, $D$ always denotes an integral domain with quotient field $K$.
Let ${\bf F}(D)$ be the set of nonzero fractional ideals of $D$.
For $I \in {\bf F}(D)$,
set $I^{-1} := \{a \in K \,|\, aI \subseteq D\}$.
An ideal $J$ of $D$ is called a {\it Glaz-Vasconcelos ideal} (for short, {\it GV-ideal}),
and denoted by $J \in {\rm GV}(D)$,
if $J$ is finitely generated and $J^{-1} = D$.
For each $I \in {\bf F}(D)$,
the {\it $w$-envelope} of $I$ is the set
$I_w : = \{a \in K \,|\, aJ \subseteq I \text{ for some $J \in {\rm GV}(D)$}\}$.
The mapping on ${\bf F}(D)$ defined by $I \mapsto I_w$ is called the {\it $w$-operation} on $D$.
An ideal $I$ of $D$ is called a {\it $w$-ideal} if
$I_w = I$
(it should be noted that $(0)$ is not a $w$-ideal because $(0)_w$ cannot be defined), and
a proper $w$-ideal $I$ of $D$ is a {\it maximal $w$-ideal} if there is no proper $w$-ideal of $D$ containing $I$.
By Zorn's lemma, it is easy to show that $D$ has at least one maximal $w$-ideal when $D$ is not a field.
An ideal $P$ of $D$ is called a {\it prime $w$-ideal} if it is both a prime ideal and a $w$-ideal.
It is routine to check that if a prime ideal $P$ of $D$ is contained in a proper $w$-ideal of $D$,
then it is a prime $w$-ideal of $D$.
Hence we can define the {\it $w$-dimension} of $D$ as $w\mhyphen\dim(D) = \sup\{{\rm ht}(M) \,|\, M \in w\mhyphen{\rm Max}(D)\}$.
At the end of this section, we state some definitions of integral domains or ideals related to the $w$-operation
used in this paper.
\begin{itemize}
\item
$D$ is a {\it DW-domain} if
every ideal of $D$ is a $w$-ideal.
\item
$D$ is a {\it UMT domain} if
every upper to zero in $D[X]$ is a maximal $w$-ideal of $D[X]$.
\item
A nonzero ideal $I$ of $D$ is {\it $w$-invertible} if
$(II^{-1})_w = D$.
\item
$D$ is a {\it Pr\"ufer $v$-multiplication domain}
(for short, {\it P$v$MD}) if
every finitely generated nonzero ideal of $D$ is $w$-invertible.
\item
$D$ is a {\it Krull domain} if
every nonzero ideal of $D$ is $w$-invertible.
\item
$D$ is a {\it strong Mori domain} if
it satisfies the ascending chain condition on $w$-ideals.
\item
$D$ has {\it finite $w$-character} if
every nonzero nonunit element belongs to only finitely many maximal $w$-ideals.
\end{itemize}
Note that a Krull domain is both a strong Mori domain and a P$v$MD,
and every strong Mori domain is a $w$-locally Noetherian domain with finite $w$-character.
Also, a UMT domain is exactly an integrally closed P$v$MD.
It is worth noting that
if $D$ is a Krull domain, then $w\mhyphen\dim(D) = 1$ \cite[Corollary 7.9.4(1)]{wang book}.

\section{Basic properties and examples of \texorpdfstring{$w$}{w}-Hilbert domains}\label{sec 2}

In this section,
we define the notion of $w$-Hilbert domains,
and investigate basic properties of $w$-Hilbert domains.
Also, we examine some examples of $w$-Hilbert domains.
Recall that an integral domain is a {\it Hilbert domain}
if every prime ideal is an intersection of maximal ideals.
Similarly, we can define the $w$-analogue of the Hilbert domain.
An integral domain is said to be a {\it $w$-Hilbert domain}
if every prime $w$-ideal can be expressed as
an intersection of maximal $w$-ideals.
We first investigate basic examples of $w$-Hilbert domains.

\begin{example}\label{w-Hilbert exam}
{\rm

(1) Every DVR is a $w$-Hilbert domain which is not a Hilbert domain.

(2) Every one-$w$-dimensional integral domain is a $w$-Hilbert domain.
Since every Krull domain is one-$w$-dimensional,
a Krull domain is always a $w$-Hilbert domain.

(3) Let $D$ be a DW-domain.
If $D$ is a Hilbert domain,
then every prime $w$-ideal of $D$ is
an intersection of maximal ideals of $D$,
so it is an intersection of maximal $w$-ideals of $D$.
Hence every Hilbert domain is a $w$-Hilbert domain
when it is a DW-domain.
On the other hand, every DVR is a DW-domain.
The first argument demonstrates that the concepts of a Hilbert domain and a $w$-Hilbert domain do not coincide
even in the case of DW-domains.

(4) Every prime $w$-ideal of a Hilbert domain with finite character is a maximal ideal.
This implies that every Hilbert domain with finite character is a $w$-Hilbert domain.

(5) For $n \geq 2$,
let $V$ be a rank $n$ valuation domain with maximal ideal $M$.
Then there exists a chain $P \subsetneq M$ of
prime ideals of $V$.
Since every valuation domain is a DW-domain,
$P$ is a prime $w$-ideal.
However, $P \neq M$,
which means that $V$ is not a $w$-Hilbert domain.
}
\end{example}

It is a well-known fact that a commutative ring with identity is a Hilbert ring if and only if 
the contraction of every maximal ideal in its polynomial ring is also a maximal ideal. 
The following example shows that the above fact for maximal $w$-ideals does not hold in $w$-Hilbert domains.

\begin{example}
{\rm
Let $V$ be a DVR with maximal ideal $M$.
Then $V$ is a $w$-Hilbert domain by Example \ref{w-Hilbert exam}(1).
Also, it is easy to show that $XV[X]$ is a maximal $w$-ideal of $V[X]$.
However, $(0) = XV[X] \cap V$ is not a maximal $w$-ideal of $V$.
Thus the contraction of a maximal $w$-ideal of the polynomial ring over a $w$-Hilbert domain is not necessarily a maximal $w$-ideal.
}
\end{example}

An ideal of an integral domain is a {\it $w$-G-ideal} if
it is both a G-ideal and a $w$-ideal.
It is a trivial fact that every maximal ideal is a G-ideal.
However, the following example highlights a striking difference in $w$-operation theory:
A maximal $w$-ideal is not necessarily a $w$-G-ideal.

\begin{example}
{\rm
Let $k$ be a field and
let $X,Y$ be indeterminates over $k$.
Consider the ring $D = k[X,Y]$.
Then $D$ is a Krull domain,
and hence it is a $w$-Hilbert domain.
Since $(Y)$ is a height-one prime ideal of $D$,
$(Y)$ is a maximal $w$-ideal of $D$.
However, $D/(Y) = k[X]$ is not a G-domain,
which means that $(Y)$ is not a $w$-G-ideal.
}
\end{example}

From now on,
we investigate basic properties of $w$-Hilbert domains.
First, we investigate the relationship between $w$-G-ideals and maximal $w$-ideals in a $w$-Hilbert domain.

\begin{proposition}\label{w-G-ideal is maximal}
Let $D$ be a $w$-Hilbert domain.
Then every $w$-G-ideal of $D$ is a maximal $w$-ideal of $D$.
\end{proposition}

\begin{proof}
Suppose that $D$ is a $w$-Hilbert domain.
Let $P$ be a $w$-G-ideal of $D$.
Then there exists a subset $\{M_{\alpha} \,|\, \alpha \in \Lambda\}$ of $w\mhyphen{\rm Max}(D)$ such that
$P = \bigcap_{\alpha \in \Lambda}M_{\alpha}$.
Suppose to the contrary that $P$ is not a maximal $w$-ideal of $D$.
Then $M_{\alpha}/P$ is a nonzero prime ideal of $D/P$ for any $\alpha \in \Lambda$.
As $D/P$ is a G-domain,
there exists a nonzero element $a+P \in (D/P) \setminus\{0+P\}$ such that
$a+P \in M_{\alpha}/P$ for all $\alpha \in \Lambda$.
Hence $a+P \in \bigcap_{\alpha \in \Lambda}(M_{\alpha}/P) = (\bigcap_{\alpha \in \Lambda}M_{\alpha})/P = \{0+P\}$.
This is a contradiction,
so $P$ is a maximal $w$-ideal of $D$.
\end{proof}

The following result concerns the interplay between
the $w$-spectrum and the maximal $w$-spectrum in a $w$-Hilbert domain.

\begin{proposition}\label{infinite maximal $w$-ideal}
Let $D$ be a $w$-Hilbert domain.
If there exists a prime $w$-ideal of $D$ which is not a maximal $w$-ideal of $D$,
then $D$ has an infinite number of maximal $w$-ideals.
\end{proposition}

\begin{proof}
Let $P$ be a prime $w$-ideal of $D$ which is not a maximal $w$-ideal of $D$.
Suppose to the contrary that $D$ has only finitely many maximal $w$-ideals,
say $M_1,\dots,M_n$.
Since $D$ is a $w$-Hilbert domain,
we may assume that there exists $1\leq k \leq n$ such that
$P = \bigcap_{1 \leq i \leq k} M_i$.
For each $1 \leq i \leq k$,
let $a_i \in M_i \setminus P$.
Then $a_1\cdots a_k + P \in (\bigcap_{1 \leq i \leq k}M_i)/P = \{0+P\}$,
so $a_i \in P$ for some $1 \leq i \leq k$,
a contradiction.
Thus $D$ has infinitely many maximal $w$-ideals.
\end{proof}

The next example shows that the converse of Proposition \ref{infinite maximal $w$-ideal} does not hold in general.

\begin{example}
{\rm
Note that $w\mhyphen\dim(\mathbb{Z}) = 1$, so
every prime $w$-ideal of $\mathbb{Z}$ is a maximal $w$-ideal and
$\mathbb{Z}$ is a $w$-Hilbert domain by Example \ref{w-Hilbert exam}.
Also, since $\mathbb{Z}$ is a DW-domain,
$\mathbb{Z}$ has an infinite number of maximal $w$-ideals.
Thus the converse of Proposition \ref{infinite maximal $w$-ideal} does not hold in general.
}
\end{example}

Next, we investigate $w$-Hilbert domains in the context of strong Mori domains.
Let $D$ be an integral domain and
let $S$ be a multiplicative subset of $D$.
Suppose that $D$ is a strong Mori domain.
Then it is routine to check that
if $P$ is a prime $w$-ideal of $D$ disjoint from $S$,
then $PD_S$ is a prime $w$-ideal of $D_S$.
This implies that if $M$ is a maximal $w$-ideal of $D$,
then $MD_M$ is the unique maximal $w$-ideal of $D_M$.
Using the above facts,
we obtain the following result.

\begin{proposition}\label{quotient extension}
Let $D$ be a strong Mori domain and
let $S$ be a multiplicative subset of $D$.
If $D$ is a $w$-Hilbert domain,
then so is $D_S$.
\end{proposition}

\begin{proof}
Suppose that $D$ is a $w$-Hilbert domain.
Let $\mathfrak{p}$ be a prime $w$-ideal of $D_S$.
Then there exists a prime $w$-ideal $P$ of $D$ such that
$\mathfrak{p} = PD_S$.
Hence
there exists a subset $\{M_{\alpha} \,|\, \alpha \in \Lambda\}$ of $w\mhyphen{\rm Max}(D)$ such that
$P = \bigcap_{\alpha \in \Lambda} M_{\alpha}$.
Consider the set
$\Gamma := \{\alpha \in \Lambda \,|\,
M_{\alpha} \cap S = \emptyset\}$.
Then $\mathfrak{p} = PD_S
= \bigcap_{\alpha \in \Lambda}M_{\alpha}D_S
= \bigcap_{\alpha \in \Gamma}M_{\alpha}D_S$,
where the second equality follows from the fact that $D$ has finite $w$-character.
Thus $D_S$ is a $w$-Hilbert domain
since $M_{\alpha}D_S$ is a maximal $w$-ideal of $D_S$ for any $\alpha \in \Gamma$.
\end{proof}

Several characterizations of the condition $w\mhyphen\dim(D)=1$ have been established.
For instance, $D$ is a UMT domain if and only if
$w\mhyphen\dim(D)=1$ when $D$ is a strong Mori domain,
and $D$ is a PIT domain if and only if $w\mhyphen\dim(D)=1$ when $D$ is a P$v$MD.
Recall that $D$ has {\it PIT} if
$D$ satisfies the principal ideal theorem, {\it i.e.},
every minimal prime ideal of $D$ over
a nonzero nonunit element has at most height one
\cite[Exercise 7.46]{wang book}.
By Proposition \ref{quotient extension}, we obtain another characterization of $w\mhyphen\dim(D)=1$
when $D$ is a strong Mori domain.

\begin{theorem}\label{w-Hil w-dim 1}
Let $D$ be a strong Mori domain.
Then $D$ is a $w$-Hilbert domain if and only if
$w\mhyphen\dim(D) = 1$.
\end{theorem}

\begin{proof}
Suppose that $D$ is a $w$-Hilbert domain and
let $P$ be a prime $w$-ideal of $D$.
Then there is a maximal $w$-ideal $M$ of $D$ containing $P$.
By Proposition \ref{quotient extension},
$D_M$ is a $w$-Hilbert domain with the unique maximal $w$-ideal $MD_M$.
Since $PD_M$ is a prime $w$-ideal of $D_M$,
$PD_M = MD_M$, and hence $P = M$.
This implies that every prime $w$-ideal of $D$ is a maximal $w$-ideal.
Thus $w\mhyphen\dim(D) = 1$.
The converse is obvious.
\end{proof}

By Theorem \ref{w-Hil w-dim 1}, we obtain

\begin{corollary}\label{w-Hilbert UMT}
Let $D$ be a strong Mori domain.
Then $D$ is a $w$-Hilbert domain if and only if $D$ is a UMT domain.
\end{corollary}

From now on, we examine some examples of $w$-Hilbert domains
using $D+M$ constructions.

By Theorem \ref{w-Hil w-dim 1},
we obtain the following meaningful example,
which shows that a Hilbert domain is not necessarily a $w$-Hilbert domain in general.

\begin{example}\label{non w-Hil Hil}
{\rm
Consider the ring
$D:= \mathbb{R} + (X,Y)\mathbb{C}[X,Y]$.
Since $\mathbb{C}[X,Y]$ is a Krull domain,
$w\mhyphen\dim(\mathbb{C}[X,Y]) = 1$.
Hence $w\mhyphen\dim(D) = 2$ \cite[Theorem 8.7.9]{wang book}.
Also, note that $D$ is a Noetherian domain \cite[Theorem 4.12]{GH 1997}, and hence
it is a strong Mori domain.
By Theorem \ref{w-Hil w-dim 1},
$D$ is not a $w$-Hilbert domain.
On the other hand, $D$ is a Hilbert domain since $\mathbb{R}$ and $\mathbb{C}[X,Y]$ are Hilbert domains
\cite[Corollary 6]{ADF 1992}.
Consequently, $D$ is a Hilbert domain which is not a $w$-Hilbert domain.
}
\end{example}

By Examples \ref{w-Hilbert exam}(1) and \ref{non w-Hil Hil},
we obtain the fact that Hilbert domains and $w$-Hilbert domains do not imply each other.
Hence the converse of Proposition \ref{w-G-ideal is maximal} does not hold in general
since every Hilbert domain is a $w$-Hilbert domain when the converse of Proposition \ref{w-G-ideal is maximal} holds.

The next example shows that the characterization of
$w\mhyphen\dim(D)=1$ in terms of $w$-Hilbert domains does not hold
for a general integral domain $D$ that is not a strong Mori domain.

\begin{example}\label{w-Hil w-dim 2}
{\rm
Consider the ring $D: = \mathbb{Z} + X\mathbb{Q}[X]$.
Then $D$ is a DW-domain \cite[Theorem 8.7.17(3)]{wang book}.
Also, according to \cite[Corollary 3.6]{ADF 1992},
$D$ is a Hilbert domain.
This implies that $D$ is a $w$-Hilbert domain
by Example \ref{w-Hilbert exam}(3).
On the other hand, $\dim(D) = 2 = w\mhyphen\dim(D)$
\cite[Theorem 8.7.9]{wang book}.
Thus there exists a $w$-Hilbert domain whose
$w$-dimension is greater than $1$.
}
\end{example}

By Examples \ref{non w-Hil Hil} and \ref{w-Hil w-dim 2},
strong Mori domains and $w$-Hilbert domains do not imply each other.

As shown in Corollary \ref{w-Hilbert UMT},
a strong Mori domain is a $w$-Hilbert domain if and only if it is a UMT domain.
The following examples, which conclude this section,
demonstrate that the strong Mori domain assumption is indispensable.
In other words, without this hypothesis, a UMT domain and a $w$-Hilbert domain do not imply each other.
Furthermore, the following examples reveal that a $w$-Hilbert domain does not imply a Krull domain in general.

\begin{example}\label{UMT and w-Hil}
{\rm
(1) Note that every valuation domain is a UMT domain.
In Example \ref{w-Hilbert exam},
we verified that for $n \geq 2$,
a rank $n$ valuation domain is never a $w$-Hilbert domain.
Hence a UMT domain is not a $w$-Hilbert domain in general.

(2) Let $t$ be an indeterminate over $\mathbb{Q}$.
Consider the ring $R:= \mathbb{Z} + X \mathbb{Q}(t)[\![X]\!]$.
Since $\mathbb{Z}$ is integrally closed in $\mathbb{Q}(t)$,
$R$ is integrally closed \cite[Proposition 2.1]{AR 1988}.
However, $\mathbb{Q}(t)$ is not the quotient field of $\mathbb{Z}$,
which means that $R$ is not a P$v$MD \cite[Theorem 4.1]{AR 1988}.
Hence $R$ is not a UMT domain since a P$v$MD is exactly an integrally closed UMT domain
\cite[Example 7.8.2]{wang book}.
On the other hand, note that
$w\mhyphen {\rm Spec}(R) = \{M_p := p\mathbb{Z} + X\mathbb{Q}(t)[\![X]\!],
X\mathbb{Q}(t)[\![X]\!] \,|\, p \in \mathbb{P}\}$,
where $\mathbb{P}$ is the set of prime numbers.
Since $X\mathbb{Q}(t)[\![X]\!] = \bigcap_{p \in \mathbb{P}}M_p$,
$R$ is a $w$-Hilbert domain.
Hence a $w$-Hilbert domain is not a UMT domain in general.
In addition, since $w\mhyphen\dim(R) = 2$,
$R$ is not a Krull domain.
This implies that the converse of the last argument of Example \ref{w-Hilbert exam}(2)
does not hold in general.
}
\end{example}

\section{\texorpdfstring{$w$}{w}-Hilbert domains arising from the polynomial ring and its extensions} \label{sec 3}

In this section, we investigate the conditions
under which the polynomial rings and some of their extensions become $w$-Hilbert domains.

Let $D$ be an integral domain.
Then we say that $D$ has {\it property $(*)$} if
every prime $w$-ideal of $D[X]$ is either
the extension of a prime $w$-ideal of $D$ or
an upper to zero maximal $w$-ideal of $D[X]$.
The following result is one of the main results in this section.

\begin{theorem}\label{w-Hilbert poly}
Let $D$ be an integral domain.
Then the following assertions are equivalent.
\begin{enumerate}
\item[(1)]
$D$ is a $w$-Hilbert domain with property $(*)$.
\item[(2)]
$D[X]$ is a $w$-Hilbert domain.
\end{enumerate}
\end{theorem}

\begin{proof}
(1) $\Rightarrow$ (2)
Suppose that $D$ is a $w$-Hilbert domain with property $(*)$.
Let $\mathfrak{p}$ be a prime $w$-ideal of $D[X]$.
If $\mathfrak{p}$ is an upper to zero maximal $w$-ideal of $D[X]$, we are done.
Now, suppose that $\mathfrak{p} = PD[X]$
for some prime $w$-ideal $P$ of $D$.
Since $D$ is a $w$-Hilbert domain,
there exists a subset $\{M_{\alpha} \,|\, \alpha \in \Lambda\}$ of
$w\mhyphen{\rm Max}(D)$ such that
$P = \bigcap_{\alpha \in \Lambda}M_{\alpha}$.
Hence $\mathfrak{p} = PD[X] = \big(\bigcap_{\alpha \in \Lambda}M_{\alpha}\big)D[X]
= \bigcap_{\alpha \in \Lambda} M_{\alpha}D[X]$.
Thus $D[X]$ is a $w$-Hilbert domain since
$\{M_{\alpha}D[X] \,|\, \alpha \in \Lambda\}$ is
a subset of $w\mhyphen{\rm Max}(D[X])$
\cite[Proposition 2.2]{FGH 1998}.

(2) $\Rightarrow$ (1)
Suppose that $D[X]$ is a $w$-Hilbert domain.
Let $P \in w\mhyphen{\rm Spec}(D)$.
Since $PD[X]$ is a prime $w$-ideal of $D[X]$,
there exists a subset $\{\mathfrak{m}_{\alpha} \,|\, \alpha \in \Lambda\}$ of $w\mhyphen{\rm Max}(D[X])$ such that
$PD[X] = \bigcap_{\alpha \in \Lambda} \mathfrak{m}_{\alpha}$.
Since $P \neq (0)$,
$\mathfrak{m}_{\alpha}$ is not an upper to zero in $D[X]$
for any $\alpha \in \Lambda$,
which means that
for each $\alpha \in \Lambda$,
there exists a maximal $w$-ideal $M_{\alpha}$ of $D$
such that $\mathfrak{m}_{\alpha} = M_{\alpha}D[X]$ \cite[Proposition 2.2]{FGH 1998}.
Hence $PD[X] = \bigcap_{\alpha \in \Lambda}M_{\alpha}D[X] = (\bigcap_{\alpha \in \Lambda}M_{\alpha})D[X]$.
This implies that $P = \bigcap_{\alpha \in \Lambda} M_{\alpha}$.
Thus $D$ is a $w$-Hilbert domain.
Now, we claim that $D$ has property $(*)$.
Let $\mathfrak{p}$ be a prime $w$-ideal of $D[X]$.
Then there exists a subset $\{\mathfrak{m}_{\alpha} \,|\, \alpha \in \Lambda\}$ of $w\mhyphen{\rm Max}(D[X])$ such that
$\mathfrak{p} = \bigcap_{\alpha \in \Lambda} \mathfrak{m}_{\alpha}$.
First, assume that for each $\alpha \in \Lambda$,
there exists a maximal $w$-ideal $M_{\alpha}$ of $D$ such that
$\mathfrak{m}_{\alpha} = M_{\alpha}D[X]$.
Then $\mathfrak{p} = \bigcap_{\alpha \in \Lambda} M_{\alpha}D[X] = (\bigcap_{\alpha \in \Lambda} M_{\alpha})D[X]$.
This implies that $\mathfrak{p}$ is the extension of a prime $w$-ideal of $D$.
Now, assume that there exists $\alpha \in \Lambda$ such that
$\mathfrak{m}_{\alpha} \cap D = (0)$.
Since $\mathfrak{p} \subseteq \mathfrak{m}_{\alpha}$ and
$\mathfrak{m}_{\alpha}$ is an upper to zero maximal $w$-ideal of $D[X]$,
we obtain that $\mathfrak{p}$ is an upper to zero maximal $w$-ideal of $D[X]$.
Thus $D$ is a $w$-Hilbert domain with property $(*)$.
\end{proof}

By Theorem \ref{w-Hilbert poly},
we obtain

\begin{corollary}\label{w-dim over w-Hilbert}
Let $D$ be an integral domain.
Suppose that $D[X]$ is a $w$-Hilbert domain.
Then $w\mhyphen \dim(D) = w\mhyphen \dim(D[X])$.
\end{corollary}

\begin{proof}
Let $w\mhyphen \dim(D) = n$.
Suppose to the contrary that $w \mhyphen \dim(D[X]) \geq n+1$.
Then there exists a chain
$\mathfrak{p}_1 \supsetneq \mathfrak{p}_2 \supsetneq \cdots \supsetneq \mathfrak{p}_{n+1}$ of prime $w$-ideals of $D[X]$.
Note that for any $1 \leq i \leq n+1$,
$\mathfrak{p}_i$ is not an upper to zero maximal $w$-ideal of $D[X]$.
This implies that for each $1 \leq i \leq n+1$,
there exists a prime $w$-ideal $P_i$ of $D$ such that
$\mathfrak{p}_i = P_iD[X]$ by Theorem \ref{w-Hilbert poly}.
Hence we obtain the chain $P_1 \supsetneq P_2 \supsetneq \cdots \supsetneq P_{n+1}$ of prime $w$-ideals of $D$.
This contradicts the fact that $w \mhyphen \dim(D) = n$.
Thus $w \mhyphen \dim(D[X]) = n = w \mhyphen \dim (D)$.
\end{proof}

Following the characterization of $w$-Hilbert polynomial rings,
we now turn to certain extensions of polynomial rings
to determine when such extensions are $w$-Hilbert domains.
We begin with the following construction:
Let $A = \{f \in D[X] \,|\, f(0) = 1\}$.
Then $A$ is a multiplicative subset of $D[X]$,
and hence we obtain the ring $D[X]_A$,
which is called the {\it Anderson ring} of $D$.

In 2024, Baek and Lim proved that
$D[X]_A$ is never a Hilbert ring \cite[Proposition 2.11]{BL 2024}.
Motivated by this result,
it is natural to ask when
the Anderson ring is a $w$-Hilbert domain. 
To obtain the answer to this question,
we need the following lemmas.

\begin{lemma}\label{property * Anderson}
Let $D$ be an integral domain.
Then the following assertions are equivalent.
\begin{enumerate}
\item[(1)]
$D$ has property $(*)$.
\item[(2)]
Every prime $w$-ideal of $D[X]_A$ is either
the extension of a prime $w$-ideal of $D$ or
the extension of an upper to zero maximal $w$-ideal of $D[X]$ disjoint from $A$.
\end{enumerate}
\end{lemma}

\begin{proof}
(1) $\Rightarrow$ (2)
This implication follows directly from \cite[Lemma 1.2(2)]{KKL 2014}.

(2) $\Rightarrow$ (1)
Suppose that the assertion (2) holds.
Let $\mathfrak{p}$ be a prime $w$-ideal of $D[X]$.
Then $\mathfrak{p}$ satisfies one of the following conditions:
\begin{itemize}
\item[(i)]
$\mathfrak{p}$ is an upper to zero maximal $w$-ideal,
\item[(ii)]
$\mathfrak{p}$ is an upper to zero which is not a maximal $w$-ideal,
\item[(iii)]
$\mathfrak{p} \cap D \neq (0)$ and $\mathfrak{p} \cap A = \emptyset$,
\item[(iv)]
$\mathfrak{p} \cap D \neq (0)$ and $\mathfrak{p} \cap A \neq \emptyset$.
\end{itemize}
If $\mathfrak{p}$ satisfies the condition (i),
then we are done.
If $\mathfrak{p}$ satisfies the condition (ii),
then for any $f \in \mathfrak{p}$, $c(f)_w \subsetneq D$ \cite[Theorem 1.4]{HZ 1989}.
Hence $c(f) \subsetneq D$, which means that $\mathfrak{p}$ is disjoint from $A$.
Since $\mathfrak{p} D[X]_A$ is a height-one prime ideal of $D[X]_A$,
it is a prime $w$-ideal of $D[X]_A$.
This contradicts our assumption.
Hence $\mathfrak{p}$ cannot satisfy the condition (ii).
Now, assume that $\mathfrak{p}$ satisfies the condition (iii).
Then $\mathfrak{p}D[X]_A$ is a prime $w$-ideal, and
it cannot be expressed as the extension of an upper to zero maximal $w$-ideal of $D[X]$ since $\mathfrak{p} \cap D \neq (0)$.
Hence there exists a prime $w$-ideal $P$ of $D$ such that
$\mathfrak{p}D[X]_A = PD[X]_A$,
so $\mathfrak{p} = PD[X]$.
Hence $\mathfrak{p}$ is the extension of a prime $w$-ideal of $D$.
Lastly, suppose that $\mathfrak{p}$ satisfies the condition (iv).
Then $(\mathfrak{p} \cap D)D[X] \subsetneq \mathfrak{p}$,
so there exists an upper to zero $\mathfrak{q}$ in $D[X]$ such that $\mathfrak{q} \subsetneq \mathfrak{p}$ \cite[Theorem 7.3.25]{wang book}.
This implies that $\mathfrak{q}$ is not a maximal $w$-ideal.
Hence $\mathfrak{q}$ satisfies the condition (ii), a contradiction.
Therefore $\mathfrak{p}$ cannot satisfy the condition (iv).
Thus $D$ has property $(*)$.
\end{proof}

\begin{lemma}\label{inter distri}
Let $D$ be an integral domain and
let $\{I_{\alpha} \,|\, \alpha \in \Lambda\}$ be a set of proper ideals of $D$.
Then $(\bigcap_{\alpha \in \Lambda} I_{\alpha}) D[X]_A
= \bigcap_{\alpha \in \Lambda}I_{\alpha}D[X]_A$.
\end{lemma}

\begin{proof}
Let $\frac{f}{g} \in \bigcap_{\alpha \in \Lambda}I_{\alpha}D[X]_A$,
where $f \in D[X]$ and $g \in A$.
Then for any $\alpha \in \Lambda$,
$\frac{f}{g} \in I_{\alpha}D[X]_A$,
so $f \in I_{\alpha}D[X]_A$.
Hence there exists $f_{\alpha} \in A$ such that $ff_{\alpha} \in I_{\alpha}D[X]$.
For $h \in D[X]$,
let $c(h)$ be the ideal of $D$ generated by the coefficients of $h$.
Then we have
$c(f) = c(ff_{\alpha}) \subseteq I_{\alpha}$,
where the equality follows directly from \cite[Theorem 1.7.16]{wang book}
since $c(f_{\alpha}) = D$.
Hence $c(f) \subseteq \bigcap_{\alpha \in \Lambda}I_{\alpha}$.
Therefore $f \in c(f)D[X] \subseteq (\bigcap_{\alpha \in \Lambda}I_{\alpha})D[X]$.
Thus $\frac{f}{g} \in (\bigcap_{\alpha \in \Lambda}I_{\alpha})D[X]_A$.
The reverse containment is obvious.
\end{proof}

The following result shows that if the Anderson ring is a $w$-Hilbert domain,
then the structure of prime $w$-ideals becomes simpler.

\begin{theorem}\label{w-Hilbert Anderson}
Let $D$ be an integral domain.
Then the following assertions are equivalent.
\begin{itemize}
\item[(1)]
$D$ is a $w$-Hilbert domain and
$D$ has property $(*)$.
\item[(2)]
$D[X]_A$ is a $w$-Hilbert domain.
\end{itemize}
\end{theorem}

\begin{proof}
(1) $\Rightarrow$ (2)
Suppose that $D$ is a $w$-Hilbert domain and $D$ has property $(*)$.
Then $D[X]$ is a $w$-Hilbert domain by Theorem \ref{w-Hilbert poly}.
Let $\mathfrak{p}$ be a prime $w$-ideal of $D[X]_A$.
Then there exists a prime $w$-ideal $\mathfrak{q}$ of $D[X]$ disjoint from $A$ such that
$\mathfrak{p} = \mathfrak{q}D[X]_A$.
If $\mathfrak{q}$ is a maximal $w$-ideal of $D[X]$,
then $\mathfrak{p}$ is a maximal $w$-ideal of $D[X]_A$ \cite[Theorem 4.5]{BL 2024},
and hence we are done.
Now, assume that $\mathfrak{q}$ is not a maximal $w$-ideal of $D[X]$.
Then there exists a subset $\{\mathfrak{m}_{\alpha} \,|\, \alpha \in \Lambda\}$
of $w \mhyphen {\rm Max}(D[X])$ such that
$\mathfrak{q} = \bigcap_{\alpha \in \Lambda} \mathfrak{m}_{\alpha}$.
It is obvious that $\mathfrak{m}_{\alpha}$ is not an upper to zero in $D[X]$ for any $\alpha \in \Lambda$,
which means that for each $\alpha \in \Lambda$,
there exists $M_{\alpha} \in w \mhyphen {\rm Max}(D)$ such that
$\mathfrak{m}_{\alpha} = M_{\alpha}D[X]$
\cite[Proposition 2.2]{FGH 1998}.
This implies that $\mathfrak{q} = \bigcap_{\alpha \in \Lambda} M_{\alpha}D[X]
= (\bigcap_{\alpha \in \Lambda} M_{\alpha}) D[X]$.
By Lemma \ref{inter distri}, we obtain
\begin{center}
$\mathfrak{p} = \mathfrak{q}D[X]_A = (\bigcap_{\alpha \in \Lambda} M_{\alpha})D[X]_A
= \bigcap_{\alpha \in \Lambda} M_{\alpha}D[X]_A.$
\end{center}
According to \cite[Theorem 4.5]{BL 2024},
$\{M_{\alpha}D[X]_A \,|\, \alpha \in \Lambda\}$ is
a subset of $w\mhyphen{\rm Max}(D[X]_A)$.
Thus $D[X]_A$ is a $w$-Hilbert domain.

(2) $\Rightarrow$ (1)
The proof of this implication is omitted,
as it is analogous to the proof of (2) $\Rightarrow$ (1) in Theorem \ref{w-Hilbert poly}.
\end{proof}

Theorem \ref{w-Hilbert Anderson} provides a framework to construct $w$-Hilbert domains that are
neither Hilbert domains, DW-domains, nor DVRs.

\begin{example}\label{w-Hilbert nor DVR}
{\rm
Let $V$ be a DVR with maximal ideal $M$.
Then $V$ is a $w$-Hilbert domain.
Also, note that $w\mhyphen\dim(V[X])=1$,
which means that every prime $w$-ideal of $V[X]$ is a maximal $w$-ideal of $V[X]$.
Hence $V$ has property $(*)$.
By Theorem \ref{w-Hilbert Anderson},
$V[X]_A$ is a $w$-Hilbert domain.
On the other hand, $V[X]_A$ is neither a Hilbert domain, a DW-domain nor a valuation domain
\cite[Proposition 2.12, Corollary 3.7 and Theorem 4.5]{BL 2024}.
}
\end{example}

According to Theorems \ref{w-Hilbert poly} and \ref{w-Hilbert Anderson},
we can obtain several immediate consequences.

\begin{corollary}\label{infinite maximal $w$ poly}
Let $D$ be an integral domain.
Suppose that $D[X]$ is a $w$-Hilbert domain.
If one of $D[X]$ or $D[X]_A$ has a prime $w$-ideal which is not a maximal $w$-ideal,
then $D$ has an infinite number of maximal $w$-ideals.
\end{corollary}

\begin{proof}
Let $D[X]$ be a $w$-Hilbert domain.
Then $D$ has property $(*)$
by Theorem \ref{w-Hilbert poly}.
Hence if $D[X]$ has a prime $w$-ideal $\mathfrak{p}$
which is not a maximal $w$-ideal,
then $\mathfrak{p} = PD[X]$ for some prime $w$-ideal $P$ of $D$.
Since every maximal $w$-ideal of $D[X]$ containing $\mathfrak{p}$ is the extension of a maximal $w$-ideal of $D$,
$P$ is not a maximal $w$-ideal of $D$.
Thus $D$ has an infinite number of maximal $w$-ideals
by Proposition \ref{infinite maximal $w$-ideal}.
According to \cite[Theorem 4.5]{BL 2024} and
Theorem \ref{w-Hilbert Anderson},
the argument also holds for $D[X]_A$.
\end{proof}

The next example shows that the converse of Corollary \ref{infinite maximal $w$ poly} does not hold in general.

\begin{example}
{\rm
Consider the ring $\mathbb{Z}$.
Since $\mathbb{Z}$ is a Krull domain,
$\mathbb{Z}[X]$ is also a Krull domain,
and hence
$w\mhyphen\dim(\mathbb{Z}[X]) = 1$.
This implies that $\mathbb{Z}[X]$ is a $w$-Hilbert domain and
$w\mhyphen{\rm Max}(\mathbb{Z}[X]) = w\mhyphen{\rm Spec}(\mathbb{Z}[X])$,
which means that $w\mhyphen{\rm Max}(\mathbb{Z}[X]_A) = w\mhyphen{\rm Spec}(\mathbb{Z}[X]_A)$.
On the other hand, we already know that $\mathbb{Z}$ has an infinite number of maximal $w$-ideals.
Thus the converse of Corollary \ref{infinite maximal $w$ poly} does not hold in general.
}
\end{example}

We now turn to another extension of polynomial rings
to investigate when it is a $w$-Hilbert domain.
To construct such an extension, we recall the following definition:
The {\it content ideal} of $f \in D[X]$
is the ideal of $D$ generated by the coefficients of $f$,
and denoted by $c(f)$.
Consider the set $N_v = \{f \in D[X] \,|\, c(f)^{-1} = D\}$.
Then $N_v$ is a multiplicative subset of $D[X]$,
so we obtain the ring $D[X]_{N_v}$.
We refer to $D[X]_{N_v}$ as the {\it $t$-Nagata ring} of $D$.

A well-known fact is that $D[X]_N$ is a Hilbert domain if and only if
$D$ is a Hilbert domain and every prime ideal of $D[X]_N$ is the extension of a prime ideal of $D$,
where $N = \{f \in D[X] \,|\, c(f) = D\}$
\cite[Lemma 4.1]{AAM 1985}.
Since the $w$-versions of Hilbert domains and $D[X]_N$ are
$w$-Hilbert domains and $D[X]_{N_v}$, respectively,
we expect a similar characterization for the $w$-Hilbert $t$-Nagata ring.
To do this, we need the following two lemmas.

\begin{lemma}\label{t-Nagata lemma}
Let $D$ be an integral domain and
let $I$ be a proper ideal of $D$.
Then the following assertions hold.
\begin{enumerate}
\item[(1)]
$ID[X]_{N_v} = I_wD[X]_{N_v} = (ID[X]_{N_v})_w$.
\item[(2)]
$ID[X]_{N_v} \cap D = I_w$.
\end{enumerate}
\end{lemma}

\begin{proof}
(1) We first claim that $I_wD[X]_{N_v} \subseteq (ID[X]_{N_v})_w$.
Let $\frac{f}{g} \in I_wD[X]_{N_v}$ with $f \in I_wD[X]$ and $g \in N_v$.
Then $c(f) \subseteq I_w$,
so there exists $J \in GV(D)$ such that $c(f)J \subseteq I$.
Hence we obtain $c(f)JD[X]_{N_v} \subseteq ID[X]_{N_v}$.
According to \cite[Proposition 2.2(1)]{kang 1989},
$JD[X]_{N_v} \in GV(D[X]_{N_v})$,
which means that $c(f) \subseteq (ID[X]_{N_v})_w$.
Therefore $\frac{f}{g} \in (ID[X]_{N_v})_w$,
and hence our claim holds.
Thus we obtain
\begin{center}
$ID[X]_{N_v} \subseteq I_wD[X]_{N_v} \subseteq (ID[X]_{N_v})_w = ID[X]_{N_v}$,
\end{center}
where the last equality follows directly from the fact that $D[X]_{N_v}$ is a DW domain
\cite[Theorem 6.6.18]{wang book}.
This completes the proof.

(2) This follows directly from the assertion (1) and \cite[Proposition 2.8(1)]{kang 1989}.
\end{proof}

\begin{lemma}\label{inter distri t-Nagata}
Let $D$ be an integral domain and
let $\{I_{\alpha} \,|\, \alpha \in \Lambda\}$ be a set of proper $w$-ideals of $D$.
Then $(\bigcap_{\alpha \in \Lambda} I_{\alpha}) D[X]_{N_v}
= \bigcap_{\alpha \in \Lambda}I_{\alpha}D[X]_{N_v}$.
\end{lemma}

\begin{proof}
Let $\frac{f}{g} \in \bigcap_{\alpha \in \Lambda}I_{\alpha}D[X]_{N_v}$,
where $f \in D[X]$ and $g \in N_v$.
Then for any $\alpha \in \Lambda$,
$\frac{f}{g} \in I_{\alpha}D[X]_{N_v}$,
so $f \in I_{\alpha}D[X]_{N_v}$.
Hence there exists $f_{\alpha} \in N_v$ such that $ff_{\alpha} \in I_{\alpha}D[X]$,
so $c(ff_{\alpha}) \subseteq I_{\alpha}$.
According to \cite[Lemma 6.6.1]{wang book}, we have
\begin{center}
$c(f) \subseteq c(f)_w = c(ff_{\alpha})_w \subseteq (I_{\alpha})_w = I_{\alpha}$
\end{center}
since $c(f_{\alpha})_w = D$ and $I_{\alpha}$ is a $w$-ideal of $D$.
Hence $c(f) \subseteq \bigcap_{\alpha \in \Lambda}I_{\alpha}$.
Therefore $f \in c(f)D[X] \subseteq (\bigcap_{\alpha \in \Lambda}I_{\alpha})D[X]$.
Thus $\frac{f}{g} \in (\bigcap_{\alpha \in \Lambda}I_{\alpha})D[X]_{N_v}$.
The reverse containment is obvious.
\end{proof}

We are now ready to establish the last main result of this section
which characterizes when the $t$-Nagata ring is a $w$-Hilbert domain.

\begin{theorem}\label{w-Hilbert t-Nagata}
Let $D$ be an integral domain.
Then the following assertions are equivalent.
\begin{enumerate}
\item[(1)]
$D$ is a $w$-Hilbert domain and
every prime ideal of $D[X]_{N_v}$ is the extension of a prime ideal of $D$.
\item[(2)]
$D[X]_{N_v}$ is a $w$-Hilbert domain.
\end{enumerate}
\end{theorem}

\begin{proof}
(1) $\Rightarrow$ (2)
Suppose that the assertion (1) holds.
Let $\mathfrak{p}$ be a prime $w$-ideal of $D[X]_{N_v}$.
Then there exists a prime ideal $P$ of $D$ such that $\mathfrak{p} = PD[X]_{N_v}$,
so $\mathfrak{p} = P_wD[X]_{N_v}$ by Lemma \ref{t-Nagata lemma}(1).
According to \cite[Proposition 1.1]{WC 1999},
$P$ is a prime $w$-ideal of $D$.
Hence there exists a subset $\{M_{\alpha} \,|\, \alpha \in \Lambda\}$ of $w\mhyphen{\rm Max}(D)$ such that
$P = \bigcap_{\alpha \in \Lambda} M_{\alpha}$,
so we obtain
\begin{center}
$\mathfrak{p} = PD[X]_{N_v} = (\bigcap_{\alpha \in \Lambda}M_{\alpha})D[X]_{N_v} = \bigcap_{\alpha \in \Lambda} M_{\alpha}D[X]_{N_v}$,
\end{center}
where the last equality follows directly from Lemma \ref{inter distri t-Nagata}.
Note that $D[X]_{N_v}$ is a DW-domain \cite[Theorem 6.6.18]{wang book},
so every maximal $w$-ideal of $D[X]_{N_v}$ is of the form
$MD[X]_{N_v}$, where $M \in w\mhyphen{\rm Max}(D)$
\cite[Proposition 2.1(2)]{kang 1989}.
Thus $D[X]_{N_v}$ is a $w$-Hilbert domain.

(2) $\Rightarrow$ (1)
Suppose that $D[X]_{N_v}$ is a $w$-Hilbert domain.
Let $P$ be a prime $w$-ideal of $D$.
Then $PD[X]_{N_v}$ is a prime $w$-ideal of $D[X]_{N_v}$
by Lemma \ref{t-Nagata lemma}(1).
Hence there exists a subset $\{\mathfrak{m}_{\alpha} \,|\, \alpha \in \Lambda\}$ of maximal $w$-ideals of $D[X]_{N_v}$ such that
$PD[X]_{N_v} = \bigcap_{\alpha \in \Lambda} \mathfrak{m}_\alpha$.
Since $D[X]_{N_v}$ is a DW domain and
every maximal $w$-ideal of $D[X]_{N_v}$ is of the form $MD[X]_{N_v}$
for some $M \in w\text{-Max}(D)$ \cite[Theorem 6.6.18]{wang book} \cite[Proposition 2.1(2)]{kang 1989},
respectively,
for each $\alpha \in \Lambda$,
there exists a maximal $w$-ideal $M_\alpha$ of $D$ such that $\mathfrak{m}_\alpha = M_\alpha D[X]_{N_v}$.
Hence we obtain $PD[X]_{N_v} = \bigcap_{\alpha \in \Lambda} M_\alpha D[X]_{N_v} = (\bigcap_{\alpha \in \Lambda} M_\alpha)D[X]_{N_v}$,
where the last equality follows directly from Lemma \ref{inter distri t-Nagata}.
By Lemma \ref{t-Nagata lemma}(2), we have
$P = \bigcap_{\alpha \in \Lambda} M_\alpha$
since every intersection of $w$-ideals is again a $w$-ideal.
Thus $D$ is a $w$-Hilbert domain.
For the remaining argument, let $\mathfrak{p}$ be a prime $w$-ideal of $D[X]_{N_v}$.
Then there exists a set $\{\mathfrak{m}_\alpha \,|\, \alpha \in \Lambda\}$ of maximal $w$-ideals of $D[X]_{N_v}$ such that
$\mathfrak{p} = \bigcap_{\alpha \in \Lambda} \mathfrak{m}_\alpha$.
Similar to the above, we can take a subset
$\{M_{\alpha} \,|\, \alpha \in \Lambda\}$ of $w\text{-Max}(D)$ such that
$\mathfrak{m}_{\alpha} = M_{\alpha}D[X]_{N_v}$
\cite[Proposition 2.1(2)]{kang 1989}.
Hence we have
$\mathfrak{p} = \bigcap_{\alpha \in \Lambda} \mathfrak{m}_{\alpha} = \bigcap_{\alpha \in \Lambda} M_{\alpha} D[X]_{N_v} = (\bigcap_{\alpha \in \Lambda} M_\alpha)D[X]_{N_v}$,
where the last equality follows directly from Lemma \ref{inter distri t-Nagata}.
Thus $\mathfrak{p}$ is the extension of a prime ideal of $D$.
\end{proof}

At the end of this paper,
note that $D$ is a P$v$MD if and only if
$D$ is integrally closed and every prime ideal of $D[X]_{N_v}$ is extended from $D$
\cite[Corollary 3.16]{kang 1989}.
Combining this characterization with Theorem \ref{w-Hilbert t-Nagata},
we obtain the following result.

\begin{corollary}\label{PvMD w-Hilbert}
Let $D$ be an integrally closed domain.
Then $D$ is both a $w$-Hilbert domain and a P$v$MD if and only if
$D[X]_{N_v}$ is a $w$-Hilbert domain.
\end{corollary}

%
%
%

%
%

\end{document}